\documentclass{amsart}

\usepackage[english]{babel}

\usepackage{mathrsfs, mathtools, amssymb}

\usepackage{dsfont}
\usepackage{stmaryrd}
\usepackage{tikz, tikz-cd}
\usepackage[paper=a4paper, margin=2.5cm]{geometry}

\usepackage{hyperref}
\hypersetup{
    unicode=false,          
    pdftoolbar=true,        
    pdfmenubar=true,        
    pdffitwindow=false,     
    pdfstartview={FitH},    
    pdftitle={Cohomology Rings of Toric bundles and the ring of conditions},    
    pdfauthor={Hofscheier, J.; Khovanskii, A. G.; Monin, L.},     
    pdfkeywords={Toric bundles} {Spherical Varieties} {Newton Polyhedra}  {Newton Okounkov bodies}, 
    pdfnewwindow=true,      
    colorlinks=true,       
    linkcolor=blue,          
    citecolor=red,        
    filecolor=magenta,      
    urlcolor=cyan           
}

\theoremstyle{plain}
\newtheorem{theorem}{Theorem}[section]
\newtheorem{lemma}[theorem]{Lemma}
\newtheorem{proposition}[theorem]{Proposition}

\theoremstyle{definition}
\newtheorem{definition}[theorem]{Definition}

\newtheorem{example}[theorem]{Example}
\newtheorem{conjecture}[theorem]{Conjecture}

\newcommand{\rleft}{\mathopen{}\mathclose\bgroup\left}
\newcommand{\rright}{\aftergroup\egroup\right}

\newcommand{\C}{{\mathbb{C}}}

\newcommand{\R}{{\mathbb{R}}}
\newcommand{\Z}{{\mathbb{Z}}}
\newcommand{\p}{{\mathbb{P}}}

\newcommand{\Om}{{\mathcal{O}}}
\newcommand{\Pm}{{\mathcal{P}}}

\newcommand{\cL}{{\mathcal{L}}}
\newcommand{\cT}{{\mathcal{T}}}

\newcommand{\diff}{\mathop{}\!d}

\newcommand{\Hom}{{\mathrm{Hom}}}

\newcommand{\Pic}{{\mathrm{Pic}}}

\newcommand{\Vol}{{\mathrm{Vol}}}

\usepackage{pgf,tikz}
\usepackage{mathrsfs}
\usetikzlibrary{arrows}
\definecolor{uuuuuu}{rgb}{0.26666666666666666,0.26666666666666666,0.26666666666666666}
\definecolor{yqyqyq}{rgb}{0.5019607843137255,0.5019607843137255,0.5019607843137255}
\definecolor{uququq}{rgb}{0.25098039215686274,0.25098039215686274,0.25098039215686274}

\begin{document}
\selectlanguage{english}

\title{Fibered toric varieties}
\dedicatory{To Yulij Sergeievich Ilyashenko on the occasion of his 80-th birthday}
\author{Askold Khovanskii}
\address[A.\,Khovanskii]{Department of Mathematics, University of Toronto, Toronto, Canada; Moscow Independent University, Moscow, Russia.}
\email{askold@math.utoronto.ca}
\author{Leonid Monin}
\address[L.\,Monin]{Institute of Mathematics, EPFL, Lausanne, Switzerland.}
\email{leonid.monin@epfl.ch}


\begin{abstract}
A toric variety is called fibered if it can be represented as a total space of fibre bundle over toric base and with toric fiber. Fibered toric varieties form a special case of toric variety bundles. In this note we first give an introduction to the class of fibered toric varieties. 
Then we use them to illustrate some known and conjectural results on topology and intersection theory of general toric variety bundles.
Finally, using the language of fibered toric varieties, we compute the equivariant cohomology rings of smooth complete toric varieties.
\end{abstract}

\maketitle




\section{Introduction}
This paper is devoted to the study of fibered toric varieties. A fibered toric variety is the total space of a fiber bundle over toric base with a toric fiber. A first example of fibered toric variety is a Hirzebruch surface $F_a=\p(\Om_{\p^1}\oplus\Om_{\p^1}(a))$ which forms a $\p^1$-bundle over $\p^1$. More generally, for every split vector bundle $E= \cL_1\oplus\ldots\oplus \cL_k$ on a toric variety $X_\Sigma$ its projectivisation $\p(E)$ is a fibered toric variety. Another example of fibered toric varieties are so-called Bott towers \cite{karsh} which are chains of $\p^1$-bundles over~$\p^1$. 

As usual in toric geometry, the fibered property could be formulated combinatorially. More concretely, on the level of fans fibered toric varieties are classified by fans which can be represented as a twisted product (see Definintion~\ref{def:twistedprod}). On the level of polytopes, this corresponds to linear families of polytopes as in \cite[Definition 1.2]{kavvil} (see also Section~10.3 in \cite{hofscheier2023cohomology}).

Our motivation to study fibered toric varieties is twofold. On one hand, fibered toric varieties form an important class of toric varieties. They naturally appear in   classifications of smooth Fano polytopes with many vertices 
\cite{obro2007classification,obro2008classification,assarf2014smooth} or smooth polytopes of small degree with respect to its dimension \cite{batyrev2007multiples,dickenstein2009classifying}. A construction similar to twisted fans also appears in \cite{joswig2005one} where the authors construct simplicial complexes with large automorphism groups and linear families of polytopes appear in spherical geometry \cite{brion1989groupe}. The language of fibered toric varieties was useful in \cite{lindberg2023algebraic} where the authors computed the algebraic degree of large class of sparse polynomial optimisation problems.

On the other hand, fibered toric varieties are examples of more general toric variety bundles. Toric variety bundles are (partial) equivariant compactifications of a principal torus bundle. One  actively studied class of toric variety bundles is toroidal horospherical varieties, which are toric variety bundles over generalized flag varieties $G/P$. See \cite{Knop:SphEmbd, Timashev:HSEE} for more details. 
Toric variety bundles also appear in the logarithmic Gromov-Witten theory \cite{carocci2021rubber,carocci2022tropical}. 

Toric variety bundles sometimes called just toric bundles (as for example in \cite{US03,hofscheier2023cohomology}). However, we decided to adapt the term toric variety bundles (used for instance in \cite{carocci2021rubber,carocci2022tropical}) to resolve the confusion with toric vector bundles (equivariant vector bundles over toric varieties) \cite{klyachko1990equivariant,payne2008moduli}.

In Theorem~\ref{thm:torbuntor} we show that a toric variety bundle over a toric base has a toric structure and thus is a fibered toric variety. We use this together with the combinatorial description of fibered toric varieties by twisted product of fans to illustrate recent topological results on general toric variety bundles. 

In particular, in Theorem~\ref{toricUS} we describe the cohomology ring of fibered toric variety illustrating a result of \cite{US03} and in Theorem~\ref{toricBKK} we give a version of Bernstein-Kushnirenko-Khovanskii (BKK) theorem which computes the intersection number of divisors on fibered toric variety. This illustrates the results of \cite{hofscheier2023cohomology}. 
Further, in Subsection~\ref{sec:chern} we give a conjectural formula for the Chern classes of the tangent bundle on a toric variety bundle and verify it in the case of fibered toric varieties. We conclude by the computation of the equivariant cohomology ring of a smooth toric variety using the language of fibered toric varieties in Section~\ref{sec:eq}.

\subsection*{Acknowledgments} We would like to thank J. Hofschaier, M.~Joswig, V.~Timorin and A.~Voorhaar for useful conversations and comments on the previous versions of this text.  The first author is partially supported by the Canadian Grant No.~156833-17.

\section{Twisted product of fans and fibered toric varieties}
In this section we give a combinatorial description of fibered toric varieties.
Let $T_B\simeq(\C^*)^k, T_F\simeq (\C^*)^n$ be a pair of algebraic tori with character lattices $M_B\simeq \Z^k, M_F\simeq \Z^n$ and co-character lattices $N_B=\Hom(M_B,\Z), N_F=\Hom(M_F,\Z)$. We further denote by $T=T_B\times T_F$ the product torus with character and co-character lattices $M$ and $N$ respectively. We will denote by $\Sigma_B, \Sigma_F, \Sigma$ rational polyhedral fans in the lattices $N_B, N_F$, and $N$ respectively and by $X_{\Sigma_B}$, $X_{\Sigma_F}$, $X_{\Sigma}$ the corresponding toric varieties with respect to the corresponding torus actions. In this section we do not assume $X_{\Sigma_B}$, $X_{\Sigma_F}$, $X_{\Sigma}$ to be smooth or complete, however, throughout the paper we assume that the rays of $\Sigma_B, \Sigma_F, \Sigma$ span lattices $N_B, N_F$, and $N$ respectively.  For the details on toric geometry we refer to \cite{tor-var}.

The main object considered in this paper is a \emph{fibered toric variety.} We say that $X_\Sigma$ is a fibered $T-$toric variety if there exists a decomposition $T=T_F\times T_B$ and the natural projection $T=T_B\times T_F\to T_B$ induces a toric morphism $\pi\colon X_\Sigma\to X_{\Sigma_B}$ which forms a fibre bundle over base $X_{\Sigma_B}$ with fibre $X_{\Sigma_F}$. 

The combinatorial description of fibered toric varieties is given in terms of \emph{twisted product} of fans. The twisted product $\Sigma_B\ltimes_\Phi \Sigma_F$ of two fans $\Sigma_B,\Sigma_F$ depends on the twisting parameter $\Phi$, which is a cone-wise linear map $\Phi:|\Sigma_B| \to N_F$. More concretely, $\Phi$ is a continuous map on the support of $\Sigma_B$ such that its restriction to every cone $\sigma\in \Sigma_B$ is linear.

\begin{definition}\label{def:twistedprod}
 Let $\Sigma_B,\Sigma_F$ and  $\Phi:|\Sigma_B| \to N_F$ be as before. We define the twisted product  $\Sigma_B\ltimes_\Phi \Sigma_F$ to be a fan in $N=N_B\times N_F$ with the set of cones given by
 \begin{equation}\label{eq:twisted}
\Sigma_B\ltimes_\Phi \Sigma_F = \{ \widetilde\sigma+\tau \,|\,\sigma \in \Sigma_B, \tau \in \Sigma_F\},     
 \end{equation}
where $\widetilde \sigma = \{(x,\Phi(x))\,|\, x\in\sigma\}$ is the graph of $\Phi|_\sigma$.
\end{definition}

It is easy to see that the twisted product of fans is well defined, i.e. the set of cones in \eqref{eq:twisted} defines a rational polyhedral fan in $N$. Moreover, the natural projection $N\to N_B$ induces a morphism of fans $\Sigma=\Sigma_B\ltimes_\Phi \Sigma_F \to \Sigma_B$ and thus a toric morphism between corresponding toric varieties $\pi:X_\Sigma \to X_{\Sigma_B}$.

In the case $\Phi=0$, the twisted product recovers usual direct product of fans.
More generally, the twisted product $\Sigma_B\ltimes_\Phi \Sigma_F$ is combinatorially equivalent to a direct product $\Sigma_B\times \Sigma_F$ for any $\Phi$.

The first result which we need is the following theorem.

\begin{theorem}\label{thm:fibered}
    Let $X_\Sigma$ be a fibered toric variety. That is there is a toric map
    $p\colon X_\Sigma \to X_{\Sigma_B}$ which is a fiber bundle with fiber $X_{\Sigma_F}$ for some fans $\Sigma_B,\Sigma_F$ in $N_B, N_F$ respectively.
    Then $\Sigma =\Sigma_B\ltimes_\Phi \Sigma_F$ is a twisted product  of $\Sigma_B$ and $\Sigma_F$ for some piecewise linear map $\Phi\colon \Sigma_B \to N_F$.
 \end{theorem}
Theorem~\ref{thm:fibered} appeared in a slightly different form in \cite[Proposition~7.3]{oda1978torus}, so we skip its proof.

\begin{example}
    Let $\Sigma_B, \Sigma_F$ both be the fans of $\mathbb{P}^1$, and let $\Phi\colon \Z \to \Z$ be given by 
    \[
    \Phi(x) = \begin{cases}
    0, \quad\quad\,\, x\leq 0\\
    a\cdot x, \quad x>0
               \end{cases}
    \]
    for some $a\in \Z$. Then, the twisted product $\Sigma=\Sigma_B\ltimes_\Phi \Sigma_F$ is shown in Figure~\ref{fig:hirz} and is a fan of the Hirzebruch surface $F_a= \p(\Om_{\p^1}\oplus\Om_{\p^1}(a))$. The natural projection $\pi\colon F_a\to \p^1$ determined by the projective bundle structure is given by the map of fans $\Sigma\to \Sigma_B$ induced by the natural projection $N\to N_B$.
\end{example}

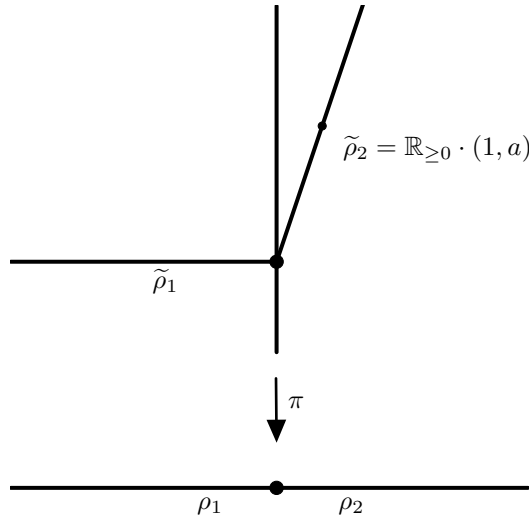
\begin{figure}
    \centering
\definecolor{cqcqcq}{rgb}{0.7529411764705882,0.7529411764705882,0.7529411764705882}
\begin{tikzpicture}[line cap=round,line join=round,>=triangle 45,x=.3cm,y=.3cm]
\clip(-9.656376051369712,-10.931241342795854) rectangle (13.022804558229346,13.306185321663618);
\draw [line width=1.5pt,domain=2:13.022804558229346] plot(\x,{(-8--6*\x)/2});
\draw [line width=1.5pt,domain=-9.656376051369712:2] plot(\x,{(-20-0*\x)/-10});
\draw [line width=1.5pt] (2,2) -- (2,13.306185321663618);
\draw [line width=1.5pt] (2,2)-- (2,-2);
\draw [line width=1.5pt,domain=1.9987355216627098:13.022804558229346] plot(\x,{(-15.990953231520804-0*\x)/2});
\draw [line width=1.5pt,domain=-9.656376051369712:1.9987355216627098] plot(\x,{(--23.98642984728116-0*\x)/-3});
\draw (-1.8904748123251864,-8) node[anchor=north west] {$\rho_1$};
\draw (4.3,-8) node[anchor=north west] {$\rho_2$};
\draw (-3.849131319517832,2.1) node[anchor=north west] {$\widetilde\rho_1$};
\draw (4.5,8.) node[anchor=north west] {$\widetilde\rho_2 = \mathbb{R}_{\geq 0}\cdot(1,a)$};
\draw [->,line width=1pt] (1.9748956573337617,-3.155745261440391) -- (2,-6);
\draw (2.095562991786163,-3.5396856933507665) node[anchor=north west] {$\pi$};
\begin{scriptsize}
\draw [fill=black] (2,2) circle (2.5pt);
\draw [fill=black] (1.9987355216627098,-7.995476615760395) circle (2.5pt);
\draw [fill=black] (4,8) circle (1.5pt);
\end{scriptsize}
\end{tikzpicture}
\caption{Hirzebruch surface $F_a$ fibered over $\mathbb{P}^1$.}
\label{fig:hirz}
\end{figure}

\begin{proposition}\label{prop:iso}
    Let $\Phi,\Phi'$ be two piecewise linear maps on $\Sigma_B$, such that $\Phi-\Phi'\in \Hom(N_B,N_F)$ is a global linear map. Then  $X_\Sigma$ and $X_{\Sigma'}$ with 
    \[
\Sigma=\Sigma_B\ltimes_\Phi \Sigma_F,\quad \Sigma=\Sigma_B\ltimes_{\Phi'} \Sigma_F
    \]
    are isomorphic as fibered toric varieties over $X_{\Sigma_B}$.
\end{proposition}
\begin{proof}
    Indeed, let 
\[
\Phi' = \Phi + \phi, \quad \phi \in \Hom(N_B,N_F).
\]
Then we have $\Sigma' = \widetilde \phi(\Sigma)$, where
\[
\widetilde{\phi}: N_B\times N_F \to N_B\times N_F, \quad \widetilde{\phi}: (v_1,v_2) \mapsto (v_1, v_2 + \phi(v_1)).
\]
Since $\widetilde\phi$ is an automorphism of the lattice $N$ which preserves $N_B$ it provides an isomorphism between $X_{\Sigma}$ and $X_{\Sigma'}$ as fibered toric varieties.
\end{proof}

\subsection{Toric variety bundles} In this subsection we briefly recall the notion of toric variety bundles. For a more detailed introduction to toric variety bundles see \cite{hofscheier2023cohomology}.

Toric variety bundles are (partial) equivariant compactifications of principal  torus bundle. More concretely, let $T$ be an algebraic torus  and let $p \colon E \to B$ be a \emph{$T$-principal bundle} over an algebraic variety $B$. Then, for every 
$T$-toric variety $X_\Sigma$, we define the associated toric variety bundle $E_\Sigma$ as 
\[
E_\Sigma = E \times_T X_\Sigma \coloneqq (E \times X_\Sigma) / T.
\]
A toric variety bundle $E_\Sigma$ comes with the natural projection $p\colon E_\Sigma\to B$ which is a fiber bundle over $B$ and with fiber $X_\Sigma$. Moreover, $E_\Sigma$ admits a $T$-action preserving the fibers of $p$.

To get a better understanding of toric variety bundles, let us give a description of $T$-principal bundles. A $T$-principle bundle $p:E\to B$ defines a group homomorphism $c\colon M\to \Pic(B)$ defined as follows. Any character $\lambda \in M$ defines a one--dimensional representation $\C_\lambda$ of $T$, namely $t \cdot z = \lambda(t) z$ for $t \in T$, and $z \in \C_\lambda$.
If $\cL_\lambda$ denotes the associated complex line bundle on $B$, i.e. $\cL_\lambda\simeq E\times_T\C_\lambda$, then $\cL_{\lambda+\mu}=\cL_\lambda\otimes\cL_\mu$, and thus we obtain a group homomorphism:
\[
c\colon M \to \Pic(B),\quad \lambda\mapsto \cL_\lambda.
\]

\begin{example}
    Consider the torus bundle $p \colon (\C^2 \setminus \{0\}) \to \mathbb{P}^1$ where $T = \C^*$ acts  diagonally on $\C^2\setminus \{0\}$, i.e., $t \cdot x = (t^{-1}x_0, t^{-1}x_1)$ for $t \in \C^*$ and $x=(x_0,x_1) \in \C^2\setminus \{0\}$.
    Let $k \in M_T = \Z$.
    Then the $T$-action on $\C_k$ is given by $t\cdot z = t^k z$.
    For $k=1$, we get that $T$ acts on $(\C^2\setminus \{0\}) \times \C_1$ via $t\cdot ((x_0,x_1),z) = ((t^{-1}x_0,t^{-1}x_1),tz)$.

    The following map induces an isomorphism between $\cL_1$ and the tautological bundle on $\mathbb{P}^1$ as it factors through the quotient $(\C^2 \setminus \{0 \})\times_{\C^*} \C_1 =((\C^2 \setminus \{0 \})\times \C_1)/\C^*$:
    \[
        (\C^2 \setminus \{0\}) \times \C_1 \to \mathbb{P}^1\times \C^2; \qquad ((x_0,x_1),z) \mapsto ([x_0:x_1], (zx_0, zx_1))
    \]
    Hence, we get that $\cL_1 = \Om_{\mathbb{P}^1}(-1)$, and thus $\cL_1=-1 \in \Pic(\p^1)=\Z$. More generally, for $k\in M=\Z$, we get $c(k) = -k \in \Pic(\p^1)=\Z$.
\end{example}

Moreover, the homomorphism $c\colon M\to \Pic(B)$ uniquely defines the principal bundle $E\to B$. Given a homomorphism $c\colon M\to \Pic(B)$, one can recover the principal bundle $E\to B$ in the following way. Let $u_1,\ldots,u_n$ be a basis of $M$, and let $\cL_i = c(u_i)$ be the corresponding line bundles on $B$ for $i=1,\ldots,n$. Let further $E\subset \cL_1\oplus\ldots\oplus\cL_k$ be the complement to coordinate vector subbundles $\cL_{i_1}\oplus\ldots\oplus\cL_{i_s}$, with $s<k$. Then the algebraic torus $T\simeq(\C^*)^k$ is acting freely on $E$ via coordinate-wise scaling, and $E\to X$. We arrive at the following result.

\begin{proposition}
  \label{prop:toricbun}
  Let $B$ be an algebraic variety and $T$ be an algebraic torus with character lattice $M$. Then $T$-principal  bundles over $B$ are in bijection with homomorphisms $c\colon M\to \Pic(B)$.
\end{proposition}

\begin{example}\label{ex:splitc}
    Let $X$ be an algebraic variety and let $\cL_1,\ldots,\cL_k$ be line bundles on $X$. Suppose $E\subset \cL_1\oplus \ldots \oplus \cL_k$ is the principal $T=(\C^*)^k$-bundle as before.
    Then for a character $\lambda = (\lambda_1, \ldots, \lambda_k) \in M_T = \Z^k$, we have $\cL_{-\lambda}= \cL_1^{\lambda_1} \otimes \ldots \otimes \cL_k^{\lambda_k}$. In other words, the homomorphism $c\colon M_T\to \Pic(X)$ is given by
    \[
        c\colon (\lambda_1,\ldots,\lambda_k) \mapsto \cL_1^{-\lambda_1} \otimes \ldots \otimes \cL_k^{-\lambda_k} \in \Pic(X).
    \]
\end{example}

\subsection{Toric variety bundles over toric base} 
In this subsection we want to study toric variety bundles $E_\Sigma$ over a toric base $B$. Our main result is the following theorem.
\begin{theorem}\label{thm:torbuntor}
    Let $T_B,T_F$ be algebraic tori. Let $B=X_{\Sigma_B}$ and $F=X_{\Sigma_F}$ be two toric varieties with respect to $T_B$ and $T_F$ actions respectively. Then every toric variety bundle over $B$ with fiber $F$ can be equipped with a structure of a $T=T_B\times T_F$-toric variety given by a fan $\Sigma = \Sigma_B\ltimes_\Phi \Sigma_F$ for some $\Phi$.
\end{theorem}
\begin{proof}
    To show the theorem, it is enough to prove that for a principal bundle $E\to X_{\Sigma_B}$ and any fan $\Sigma_F$, the corresponding toric variety bundle $E_{\Sigma_F}$ is can be equipped with a structure of a toric variety with respect to $T$. Indeed, the rest follows from Theorem~\ref{thm:fibered}. 

    To show that $E_{\Sigma_F}$ has a structure of a toric variety first let us realize the principal bundle $E$ as a complement of coordinate subbundles in $\cL_1\oplus\ldots\oplus\cL_k$. Since every line bundle on a toric variety can be equipped with an $T_B$-equivariant structure, we get that the action of $T_B$ could be extended to the action on $E$. Moreover, the action of $T_B$ on $E$ commutes with the action of $T_F\simeq (\C^*)^n$ given by coordinatewise scaling of $\cL_1\oplus\ldots\oplus\cL_k$. 
    Therefore, 
    \[
    E_{\Sigma_F} = E\times_{T_F} X_{\Sigma_F} = (E\times X_{\Sigma_F})/T_F
    \]
    inherits the action of $T_B$ from the action on $E\times X_{\Sigma_F}$ (with trivial action on $X_{\Sigma_F}$). Thus one can define an action of $T=T_B\times T_F$ on $E_{\Sigma_F}$  which makes it into a $T$-toric variety.
\end{proof}

Any two toric structures on $E_{\Sigma_F}$ constructed in the proof of Theorem~\ref{thm:torbuntor} are isomorphic. To see this, we first need to understand a picewise linear map $\Phi:\Sigma_B\to N_F$ which appears as a twisting data for the product of fans $\Sigma_F$ and $\Sigma_B$ in Theorem~\ref{thm:torbuntor}. First, $\Phi$ only depends on the principal bundle $E\to X_{\Sigma_B}$, thus we will denote it by $\Phi_E$. To construct $\Phi_E$, we will have to understand a homomorphism $c\colon M_F\to \Pic(X_{\Sigma_B})$ in more detail. 

Recall the description of the Picard group $\Pic(X_{\Sigma_B})$. Let $\rho_1,\ldots, \rho_r$ be the rays of $\Sigma_B$ and let $e_1,\ldots,e_r$ be their primitive generators. Let us further denote by $\Pm_{\Sigma_B}$ the group of integer piecewise linear functions with respect to $\Sigma_B$. Each piecewise linear function $h\in\Pm_{\Sigma_B}$ defines a line bundle $\cL_h$ in the following way:
\[
\cL_h=\Om_{X_{\Sigma_B}}(h(e_1)D_1+\ldots+h(e_r)D_r)
\]
where $D_1,\ldots,D_r$ are torus invariant divisors corresponding to rays $\rho_1,\ldots, \rho_r$ respectively. Under this correspondence, global linear functions define the trivial line bundle on $X_{\Sigma_B}$ Thus we have an isomorphism $\Pic(X_{\Sigma_B})\simeq\Pm_{\Sigma_B}/M$.

Let $E\to X_{\Sigma_B}$ be a $T_F$-principal bundle, and let $c\colon M_F\to \Pic(X_{\Sigma_B})$ be as before. Let us choose any lift $\bar{c}\colon M_F\to \Pm_{\Sigma_B}$ of the homomorphism $c$. Such a choice 
is equivalent to the  choice of compatible linearizations on line bundles $\cL_\lambda$ and thus defines a $T_B$-equivariant structure on $E$. Moreover $\bar{c}$ is unique up to a global linear map $l\in \Hom(M_F\to M_B)$. 

Then by the computation in Example~\ref{ex:splitc}, the twisting data $\Phi_E\colon\Sigma_B\to N_F$ is defined by the condition that
\begin{equation}\label{eq:phi}
    x\mapsto \Phi_E(x), \text{ such that }\langle\Phi_E(x), \lambda\rangle = -\bar{c}(\lambda)(x)\text{ for any } \lambda\in M_F.
\end{equation}
Note that since the lift $\bar{c}$ is only defined up to a global linear map $l\in \Hom(M_F\to M_B)$, the twisting data $\Phi_E(x)$ is unique up to a global homomorphism $l^\vee \in \Hom(N_B,N_F)$. Therefore, by Proposition~\ref{prop:iso}, any two twisting data for $E$ define isomorphic fibered toric variety structures on $E_{\Sigma_F}$. 

Similarly, using condition from \eqref{eq:phi}, one can recover the map $\bar{c}\colon M_F\to \Pm_{\Sigma_B}$ and thus the $T_B$-equivariant structure on $E$ from the twisting data~$\Phi$. Therefore, $T_B$-equivariant principal $T_F$ bundles over $X_{\Sigma_B}$ are classified by piecewise linear maps (with respect to $\Sigma_B$) $\Phi\colon N_B\to N_F$. This correspondence was generalized to more general equivariant principal bundles over toric varieties in \cite{kaveh2022toric,kaveh2023toric}.

\section{Intersection theory and cohomology rings of fibered toric varieties}
In this section we will illustrate some results on intersection theory of general toric variety bundles by considering a special case of fibered toric varieties. 
In this section we assume fans $\Sigma_F, \Sigma_B$ to be smooth and projective. In particular this implies that the twisted product $\Sigma=\Sigma_B\ltimes_\Phi \Sigma_F$ is also smooth and projective.
We start with the computation of cohomology rings of smooth toric variety bundles.

Let us first formulate theorem for toric variety bundles proved by Sankaran-Uma in \cite{US03}. Let $B$ be a smooth complete variety and $E\to B$ be a $T_F$-principal bundle. We will denote by $c_{top}$ the homomorphism
\[
c_{top}\colon M_F\to H^2(B,\Z),\quad  \lambda\mapsto c_1(\cL_\lambda),
\]
where $c_1(\cL_\lambda)$ is the first Chern class of a line bundle. Finally, let us denote by $\rho_1,\ldots,\rho_r$ the rays of $\Sigma_F$ and by $e_1,\ldots,e_r$ their primitive generators.

\begin{theorem}[\cite{US03}]
  \label{thm:US}
 Let $E\to B$, $\Sigma_F$  be as above and let $E_{\Sigma_F}$ be the corresponding toric variety bundle. Then the cohomology ring $H^*(E_{\Sigma_F},\R)$ is isomorphic (as an $H^*(B, \R)$-algebra) to the quotient of $H^*(B,\R)[x_1, \ldots, x_r]$ by
  \[
    \left\langle x_{j_1} \cdots x_{j_k} \colon \rho_{j_1}, \ldots, \rho_{j_k} \; \text{do not span a cone of} \; \Sigma \right\rangle + \left\langle c_{top}\left( \lambda \right) - \sum_{i=1}^n \langle e_i, \lambda \rangle x_i \colon \lambda \in M \right\rangle  \text{.}
  \]
\end{theorem}
Theorem~\ref{thm:US} is a generalization of the Stanley-Reisner description of the cohomology ring of toric varieties. Alternative descriptions of the cohomology ring of toric variety bundles were obtained in \cite{roch,hofscheier2023cohomology,khovanskii2021gorenstein}. Let us now formulate a version of Theorem~\ref{thm:US} for fibered toric varieties. For this let us further denote by $\tau_1,\ldots,\tau_s$ the rays of $\Sigma_B$ with primitive generators $f_1,\ldots, f_s$.

\begin{theorem}\label{toricUS}
Let $\Sigma_B, \Sigma_F$ be as before and let $\Sigma= \Sigma_B\ltimes_\Phi \Sigma_F$ be their twisted product for some twisting data $\Phi\colon \Sigma_B\to N_F$. Then the cohomology ring $H^*(X_{\Sigma},\R)$ of the corresponding fibered toric variety is isomorphic (as an $H^*(X_{\Sigma_B}, \R)$-algebra) to the quotient of $H^*(X_{\Sigma_B},\R)[x_1, \ldots, x_r]$ by
  \[
      \left\langle x_{j_1} \cdots x_{j_k} \colon \rho_{j_1}, \ldots, \rho_{j_k} \; \text{do not span a cone of} \; \Sigma \right\rangle + \left\langle c_{top}(\lambda) - \sum_{i=1}^r \langle e_i, \lambda \rangle x_i \colon \lambda \in M_F \right\rangle \text{.}
  \]
\end{theorem}
\begin{proof}
Our proof will be based on a classical Stanley-Reisner description of cohomology ring of toric variety. By the construction of the twisted product, $\Sigma=\Sigma_B\ltimes_\Phi \Sigma_F$ has rays $\rho_1,\ldots,\rho_r, \widetilde\tau_1,\ldots,\widetilde\tau_s$, where $\rho_i$'s are generated by $(0,e_i)\in N_B\times N_F=N$ and $\widetilde\tau_j$'s are generated by $\widetilde f_j=(f_j, \Phi(f_j))\in N$. In particular, we have:
\[
H^*(X_{\Sigma}, \R) \simeq \R[x_1,\ldots,x_r,y_1,\ldots,y_s]/ (I_{\Sigma}+J_{\Sigma}),
\]
where
\[
I_{\Sigma}=\langle x_{i_1}\ldots x_{i_l}y_{j_1}\ldots y_{j_m} \,|\, \rho_{i_1},\ldots , \rho_{i_l},\widetilde\tau_{j_1},\ldots,\widetilde\tau_{j_m} \text{ do not span a cone in } \Sigma \rangle 
\]
and 
\[
J_\Sigma = \left\langle \sum_{i=1}^r\langle (0,e_i), \lambda  \rangle x_i +\sum_{j=1}^s\langle \widetilde f_j, \lambda  \rangle y_j \:\big|\: \lambda \in M=M_B\times M_F \right\rangle.
\]

Now notice, that since the twisted product $\Sigma=\Sigma_B\ltimes_\Phi \Sigma_F$  is combinatorially equivalent to a direct product of fans, the collection of rays $\rho_{i_1},\ldots , \rho_{i_l},\widetilde\tau_{j_1},\ldots,\widetilde\tau_{j_m}$  do not span a cone in  $\Sigma$ if and only if either $\rho_{i_1},\ldots , \rho_{i_l}$ do not span a cone in $\Sigma_F$ or $\tau_{j_1},\ldots,\tau_{j_m}$ do not span a cone in $\Sigma_B$. Therefore, the ideal $I_\Sigma$ can be written as a sum $I_\Sigma = I_{\Sigma_B} + I_{\Sigma_F}$, where 
\[
I_{\Sigma_F} = \langle x_{i_1}\ldots x_{i_l} \,|\, \rho_{i_1},\ldots , \rho_{i_l} \text{ do not span a cone of } \Sigma_F \rangle,
\]
\[
I_{\Sigma_B} = \langle y_{j_1}\ldots y_{j_m} \,|\, \tau_{j_1},\ldots,\tau_{j_m} \text{ do not span a cone of } \Sigma_B \rangle.
\]

Similarly we can write the ideal $J_\Sigma$ as a sum of two ideals in the following way. If $\lambda\in M_B\subset M$, then we get $\langle (0,e_i),\lambda\rangle=0$ for any $i=1,\ldots,r$ and $\langle \widetilde f_j,\lambda\rangle=\langle  f_j,\lambda\rangle$ for every $j=1,\ldots,s$. On the other hand, if $\lambda\in M_F\subset M$, we get $\langle (0,e_i),\lambda\rangle= \langle e_i,\lambda\rangle$ and
$\langle \widetilde f_j,\lambda\rangle = \langle\Phi(f_j),\lambda\rangle$ for $j=1,\ldots,s$. Therefore,  $J_\Sigma = J_{\Sigma_B}+\widetilde J_{\Sigma_F}$, where
\[
J_{\Sigma_B} = \left\langle \sum_{j=1}^s \langle f_j, \lambda \rangle y_j \colon \lambda \in M_B\right\rangle, \quad \widetilde J_{\Sigma_F} =\left\langle \sum_{j=1}^s \langle \Phi(f_j),\lambda\rangle y_j + \sum_{i=1}^r \langle e_i, \lambda \rangle x_i \colon \lambda \in M_F \right\rangle.
\]

Notice that the generators of ideals $I_{\Sigma_B}$ and $J_{\Sigma_B}$ are contained in $\R[y_1,\ldots,y_s]$ and that the corresponding quotient ring $\R[y_1,\ldots,y_s]/(I_{\Sigma_B}+J_{\Sigma_B})$ is isomorphic to $H^*(X_{\Sigma_B})$.
Thus, we obtain the following presentation for the cohomology ring $H^*(X_\Sigma,\R)$:
\[
H^*(X_\Sigma, \R) \simeq \R[x_1,\ldots,x_r,y_1,\ldots,y_s]/ (I_{\Sigma_B}+J_{\Sigma_B})+(I_{\Sigma_F}+\widetilde J_{\Sigma_F}) =
H^*(X_{\Sigma_B},\R)[x_1, \ldots, x_r]/(I_{\Sigma_F}+\widetilde J_{\Sigma_F}). 
\]
The theorem follows from the fact that follows from the fact that the piecewise linear map $\lambda\circ\Phi$ on $\Sigma_B$ represents the line bundle $\cL_{-\lambda}$ and therefore
\[
\bar{c}(-\lambda) =\sum_{j=1}^s \langle \Phi(f_j),\lambda\rangle y_j \text{ viewed as an element of } H^2(X_{\Sigma_B},\R).
\]
This implies that $\widetilde J_{\Sigma_F}$ is given by
\[
\widetilde J_{\Sigma_F} =\left\langle c_{top}\left( \lambda \right) - \sum_{i=1}^n \langle e_i, \lambda \rangle x_i \colon \lambda \in M \right\rangle, 
\]
which finishes the proof.
\end{proof}

\subsection{BKK theorem for fibered toric varieties}
Let us now formulate a version of the Bernstein-Kushnirenko-Khovanskii theorem (BKK theorem) \cite{Kouchnirenko,Bernstein, BKK} for toric variety bundles and its version for fibered toric varieties. We start with a reminder of the classical BKK theorem.

Let $\Sigma$ be a smooth complete fan and $X_\Sigma$ the corresponding toric variety. Let $\rho_1,\ldots,\rho_r$ be rays of $\Sigma$ with primitive generators $e_1,\ldots,e_r$ and $D_1,\ldots, D_r$ be the corresponding $T$-invariant divisors. Every ample divisor $D_h= \sum_{i=1}^r h_i D_i$ on $X_\Sigma$ corresponds to a convex polytope
\[
\Delta_h =\{\lambda\big| \langle \lambda,e_i\rangle \leq h_i \}\subset M_\R.
\]
One can extend this correspondence to any divisor $D_h= \sum_{i=1}^r h_i D_i$, however, the corresponding geometric object will be a \emph{virtual polytope} $\Delta_h$. Virtual polytopes were introduced in \cite{PK92,KP}, in particular, it was shown that they have a well defined notion of volume as well and one can integrate smooth functions over virtual polytopes.

\begin{theorem}\label{thm:BKK}
    Let $X_\Sigma$ be a smooth complete toric variety of dimension $n$, and let $D_1,\ldots, D_n$ be the divisors corresponding to (possibly virtual) polytopes $\Delta_1,\ldots,\Delta_n$. Then the intersection number $D_1\cdot\ldots\cdot D_n$ can be computed as
    \[
    D_1\cdot\ldots\cdot D_n = n!\cdot\Vol(\Delta_1,\ldots,\Delta_n),
    \]
    where $\Vol$ denotes the mixed volume of the polytopes normalized with respect to $M$.
\end{theorem}

In~\cite{hofscheier2023cohomology} Theorem~\ref{thm:BKK} was generalized to the case of toric variety bundles in the following way. 
Let $B$ be a smooth complete variety of dimension $k$ and $E\to B$ be a $T_F$-principal bundle. Let further $\Sigma_F$ be a smooth complete fan and $E_{\Sigma_F}$ the corresponding toric variety bundle. Let us denote by $\rho_1,\ldots,\rho_r$ the set of rays of $\Sigma_F$.
Similar to the classical toric case, for every ray $\rho_i$ of $\Sigma_F$, toric variety bundle $E_{\Sigma_F}$ has a corresponding $T_F$-invariant divisor $\widetilde D_i\coloneqq E\times_{T_F}D_\rho$.
As in the classical toric case, we can encode a linear combination of $T_F$-invariant divisors $\widetilde D_h=\sum_{i=1}^rh_i\widetilde D_i$ by a (possibly virtual) polytope $\Delta_h$ in $(M_F)_\R$. Now we are ready to formulate a version of BKK theorem for toric variety bundles which computes the intersection indeses of divisors on $E_{\Sigma_F}$.

\begin{theorem}[\cite{hofscheier2023cohomology}]\label{thm:BKKdeg2}
    Let $p\colon E_{\Sigma_F}\to B$ be as before. Let further $\widetilde D_h=\sum h_i \widetilde D_i$ be a linear combination of $T_F$-invariant divisors and $D_B$ be a divisor on $B$. Then one can compute the self intersection index of $p^*D_B+ \widetilde D_h$ in the following way
    \[
        k!\cdot (p^*D_B+\widetilde D_h)^{n+k} = (n+k)!\cdot\int_{\Delta_h} (D_B+c(\lambda))^k \diff \mu(\lambda),
    \]
    where $\mu(\lambda)$ is the Lebesgue measure on $(M_F)_\R$ normalized with respect to $M_F$.
\end{theorem}
Let us mention that every divisor on $E_{\Sigma_F}$ is linearly equivalent to a divisor of the form $p^*D_B+\widetilde D_h$, thus Theorem~\ref{thm:BKKdeg2} completely describes the intersection theory of divisors on $E_{\Sigma_F}$.

Now let us specialize to the case of fibered toric varieties. Let $\Sigma_B, \Sigma_F$ be smooth complete fans and let $\Sigma= \Sigma_B\ltimes_\Phi \Sigma_F$ be their twisted product for some twisting data $\Phi\colon \Sigma_B\to N_F$. Let as before $\rho_1,\ldots,\rho_r$ dentote rays of $\Sigma_F$ generated by $e_1,\ldots,e_r$ and let us further denote by $\tau_1,\ldots,\tau_s$ be rays of $\Sigma_B$ with primitive generators $f_1,\ldots, f_s$.

Recall that one can recover the map $\bar c\colon M_F\to \Pm_{\Sigma_F}$ using condition~\eqref{eq:phi}:
\[
 \bar{c}(\lambda)(x)= \langle\Phi_E(x), \lambda\rangle \text{ for any } x\in N_B.
\]
In other words, the line bundle $\cL_\lambda$ on $X_{\Sigma_B}$ can be represented by a divisor:
\[
\cL_\lambda = \Om_{X_{\Sigma_B}} (D_\lambda),\quad D_\lambda= \sum_{j=1}^s \lambda\circ \Phi(f_j) D_{\rho_j}.
\]
Let us denote by $\Delta_\lambda\subset (M_B)_\R$ the (virtual) polytope corresponding to the divisor $D_\lambda$ of 
$X_{\Sigma_B}$. 

Finally, let us describe the (virtual) polytope in $M_\R$ which corresponds to a divisor $\widetilde D_h = \sum_{i=1}^r h_i\widetilde D_i$ on the fibered toric variety $X_\Sigma$.
The assignment $\lambda \mapsto \Delta_\lambda$ defines a linear family of (virtual) polytopes on $M_F$, that is $\Delta_{\lambda+\mu}=\Delta_\lambda+\Delta_\mu$. Using this family, one can define a lift of every (virtual) polytope $\Delta\subset (M_F)_\R$ to a (virtual) polytope $\widetilde \Delta$ in $M_R=(M_F)_\R\times (M_B)_\R$. If $\Delta$ is a convex polytope such that $\Delta_\lambda$ is convex for every $\lambda\in \Delta$ the lift $\widetilde \Delta$ is defined as
\[
\widetilde \Delta\coloneqq \{(\lambda,x) \in M_R=(M_F)_\R\times (M_B)_\R\:|\: \lambda\in \Delta, x\in \Delta_\lambda\} \text{.}
\]
By the construction of the lift, we get that a (virtual) polytope representing $\widetilde D_h$ is the lift $\widetilde \Delta_h$ of the polytope representing the divisor $D_h$ on the toric variety $X_{\Sigma_F}$.
For more details on the extension of the lift to the setting of virtual polytopes see \cite[Section 10.3]{hofscheier2023cohomology}.

Now we are ready to state and prove a version of Theorem~\ref{thm:BKKdeg2} for a fibered toric variety. 

\begin{theorem}\label{toricBKK}
Let $\Sigma_B, \Sigma_F$ be as before and let $\Sigma= \Sigma_B\ltimes_\Phi \Sigma_F$ be their twisted product for some twisting data $\Phi\colon \Sigma_B\to N_F$. Let further $X_\Sigma$ be the corresponding fibered toric variety, $\widetilde D_h=\sum_{i=1}^r h_i \widetilde D_{\rho_i}$ some $T_F$-invariant divisor on $X_\Sigma$ and $D_B$ some divisor on $X_{\Sigma_B}$. Then the self-intersection index of $(p^*D_B+\widetilde  D_h)$ can be computed as
\[k!\cdot (p^*D_B+\widetilde D_h)^{n+k} = (n+k)!\cdot\int_{\Delta_h} (D_B+c(\lambda))^k \diff \mu(\lambda),
\]
 where $\mu(\lambda)$ is the Lebesgue measure on $(M_F)_\R$ normalized with respect to $M_F$.
 \end{theorem}
 
 \begin{proof}
 Recall that the (virtual) polytope representing the divisor $\widetilde D_h$ on $X_\Sigma$ is given as a lift $\widetilde\Delta_h$.
  Therefore, by the classical BKK theorem applied to $X_\Sigma$ we know that the self-intersection index of $(p^*D_B+\widetilde  D_h)$ is equal to $(n+k)!\cdot\Vol(\widetilde\Delta_h+ P_B)$, where $P_B\subset (M_B)_\R\subset M_\R$ is a polytope representing divisor $D_B$ on $X_{\Sigma_B}$.

  But since $P$ is contained in a coordinate subspace, we can compute the volume $\Vol_M(\widetilde\Delta_h+ P)$ as
\[
\Vol(\widetilde\Delta_h+ P) = \int_{\Delta_h} \Vol_{M_B}(\Delta_\lambda+P)d\mu(\lambda).
\]
Moreover, by the classical BKK theorem applied to $X_{\Sigma_B}$, the integrant in the expression above is equal to $\frac{1}{k!}\cdot (c(\lambda)+D_B)^k$, which finishes the proof.
\end{proof}

\subsection{Chern classes of fibered toric varieties}\label{sec:chern} We finish this section with the conjectual formula for the Chern classes of smooth toric variety bundles which we varify in the case of fibered toric varieties.

First let us recall the classical formula for the Chern classes of tangent bundle of a smooth complete toric variety $X_{\Sigma_F}$. Let $\rho_1,\ldots,\rho_r$ be the rays of $\Sigma_{F}$ and let $D_1,\ldots D_r$ be the corresponding torus invariant divisors. We further denote by $[D_1],\ldots,[D_r]\in H^2(X_{\Sigma_F},\Z)$ the classes that are Poincar\'e dual to divisors $D_1,\ldots D_r$. One has the following exact sequence of vector bundles on $X_{\Sigma_F}$:
\[
0\to \Pic(X_\Sigma)^*\otimes_\C \Om_{X_{\Sigma_F}} \to \bigoplus_{i=1}^{r} \Om_{X_{\Sigma_F}}(D_i) \to \cT_{X_{\Sigma_F}}\to 0,
\]
where $\cT_{X_{\Sigma_F}}$ denotes a tangent bundle of $X_{\Sigma_F}$. By the Whitney sum formula we get that
$c(\cT_{X_{\Sigma_F}}) = \prod_{i=1}^r (1+t[D_i])$,
where by $c(\cT_{X_{\Sigma_F}})$ we denote the Chern polynomial:
\[
c(\cT_{X_{\Sigma_F}}) = c_0(\cT_{X_{\Sigma_F}})+c_1(\cT_{X_{\Sigma_F}})t+\ldots+c_{\dim(X_{\Sigma_F})}(\cT_{X_{\Sigma_F}})t^{\dim X}.
\]
In particular, the above calculation implies that the $i$-th Chern class of the tangent bundle is equal the sum of classes that are Poincar\'e dual to the codimension $i$ orbit closures in $X_{\Sigma_F}$:
\[
c_i(\cT_{X_{\Sigma_F}}) = \sum_{\sigma\in\Sigma, \dim\sigma=i} [O_\sigma] \in H^{2i}(X_{\Sigma_F},\Z).
\]

Now, let $E\to B$ be a $T_F$-principal bundle over a smooth complete base $B$ and let $\Sigma_F$ be a smooth complete fan. We would like to understand the Chern classes of the tangent bundle of the corresponding toric variety bundle $E_{\Sigma_F}$. Recall that for every ray $\rho_i$ one has the corresponding $T_F$-invariant divisor $\widetilde D_i$ on $E_{\Sigma_F}$. 
\begin{conjecture}\label{conj:chern}
    Let $p\colon E_{\Sigma_F}\to B$ be a toric variety bundle as before. Then the Chern polynomial of the tangent bundle of $E_{\Sigma_F}$ can be computed as
    \[
    c(\cT_{E_{\Sigma_F}}) = p^*(c(T_B))\cdot \prod_{i=1}^r(1+[\widetilde D_i]\cdot t).
    \]
\end{conjecture}

\begin{proposition}
    Conjecture~\ref{conj:chern} is true for smooth complete fibered toric varieties.
\end{proposition}
\begin{proof}
    Indeed, let $X_\Sigma$ be a fibered toric variety with $\Sigma = \Sigma_B\ltimes_\Phi \Sigma_F$. 
    Recall that $\Sigma$ has rays $\rho_i$ generated by where $(0,e_i)\in N_B\times N_F=N$ and $\widetilde\tau_j$ generated by $\widetilde f_j=(f_j, \Phi(f_j))\in N$. In particular, one has $D_{\widetilde\tau_j}=p^*(D_{\tau_j})$, where $p:X_\Sigma\to X_{\Sigma_{B}}$ is the toric fibration map. Then the Chern polynomial $c(\cT_{X_\Sigma})$ can be computed as 
    \[
    c(\cT_{X_\Sigma})= \prod_{i=1}^r(1+[D_{\rho_i}])\cdot\prod_{j=1}^s(1+[D_{\widetilde\tau_j}])= \cdot \prod_{i=1}^r(1+[D_{\rho_i}])p^*\prod_{j=1}^s(1+[D_{\tau_j}]) = \prod_{i=1}^r(1+[D_{\rho_i}])\cdot p^*(c(\cT_{X_{\Sigma_B}})),
    \]
    which finishes the proof.
\end{proof}

\section{Equivariant cohomology of toric varieties}\label{sec:eq}
In this section we will use descriptions for the cohomology ring of fibered toric varieties as in Theorem~\ref{toricUS} to compute the $T$-equivariant cohomology of toric varieties. We refer to \cite{anderson2012introduction,anderson2023equivariant} for the introduction to equivivariant cohomology.

Let $X_\Sigma$ be a smooth complete toric variety with respect to a torus $T\simeq (\C^*)^n$. Then the equivariant cohomology ring $H_T^*(X_\Sigma,\R)$ is defined as the cohomology ring of so-called Borel construction:
\[
H_T^*(X_\Sigma,\R) = H^*(ET\times_T X_\Sigma,\R),
\]
where $ET\to BT$ is the universal  $T$-principal bundle. More concretely, $ET= (\C^\infty\setminus 0)^n$ with coordinatewise action of $T\simeq (\C^*)^n$ and $BT= ET/T=(\p^\infty)^n$. In particular, $ET\times_T X_\Sigma$ is a toric variety bundle over the product of  infinite dimensional projective spaces. One would like to use Theorem~\ref{toricUS} to describe $H^*(ET\times_T X_\Sigma,\R)$. However, it is not directly applicable as the base is infinite dimensional.

To apply our Theorem~\ref{toricUS}, we use the standard approximation technique which, in particular, was used in \cite{edidin1998equivariant} to define equivariant Chow rings in algebraic geometry. The idea is to replace $ET$ and $BT$ with a sequence of finite dimensional spaces $ET_m$ and $BT_m$ which can be used to compute $H_T^i(X_\Sigma)$ for small $i$. The main tool for us is the following lemma (see for example \cite[Lemma~1.3]{anderson2012introduction}).

\begin{lemma}\label{lem:apr}
 Suppose $ET_m$ is any (connected) space with a free $T$-action,
and $H^i(ET_m,\R) = 0$ for $0 < i < k(m)$ (for some integer $k(m)$). Then for any $X$ with a $T$-action,
there are natural isomorphisms
\[
H^i(ET_m \times_T X,\R) \simeq H^i(ET \times_T X,\R) =: H^i_T(X,\R)
\]
for any $i < k(m)$.
\end{lemma}

For $T\simeq(\C^*)^n$ one can  approximate $ET,BT$ by the spaces $ET_m=(\C^{m+1}\setminus 0)^n$ and $BT_m=(\p^m)^n$. Since $H^i(ET_m,\R)=0$ for $i\leq 2m$ we get the isomorphism
\[
H^i_T(X_\Sigma,\R) \simeq H^i(ET_m\times_TX_\Sigma,\R)\text{ for any } i\leq2m.
\]
Now $ET_m\times_TX_\Sigma$ is a fibered toric variety, so we can apply Theorem~\ref{toricUS} to compute its cohomology ring. Let $\rho_1,\ldots,\rho_r$ be the rays of $\Sigma$ with primitive generators $e_1,\ldots, e_r$. Let us further denote by $H_i\in H^2(BT_m,\R)$ the pullback of the hyperplane class in the $i$-th component of $BT_m=(\p^m)^n$.
We get that $H^*(ET_m\times_TX_\Sigma,\R) = H^*(BT_m,\R)[x_1,\ldots,x_r]/I+J$,
where $I$ is the Stanley-Reisner ideal of $\Sigma$ and 
\[
J = \left\langle \sum_{i=1}^r \lambda_iH_I-\sum_{i=1}^r \langle e_i,\lambda\rangle x_i\:\big|\: (\lambda_1,\ldots,\lambda_n)\in M\simeq \Z^n\right\rangle.
\]

Since the map $c_{top}\colon M\to H^2(BT_m,\Z)$ is surjective and $H^*(BT_m,\R)\simeq R[H_1,\ldots,H_n]/\langle H_1^{m+1},\ldots,H_n^{m+1}\rangle$ we get
\[
H^*(ET_m\times_TX_\Sigma,\R) = \R[x_1,\ldots,x_r]/(I+K_m),
\]
where $I$ is the Stanley-Reisner ideal of $\Sigma$ as before and 
\[
K_m= \left\langle \left(\sum_{i=1}^r \langle e_i,\lambda_1\rangle x_i\right)^{m+1}\!\!\!\!\!\!\!\!\!\!\!,\ldots,\left(\sum_{i=1}^r \langle e_i,\lambda_n\rangle x_i\right)^{m+1}\right\rangle.
\]

Now, to compute the whole $H_T^*(X_\Sigma,\R)$ one can take the inverse limit of $H^*(ET_m\times_TX_\Sigma,\R) = \R[x_1,\ldots,x_r]/(I+K_m)$. We arrive at the following classical description of the equivivariant cohomology ring of smooth complete toric variety as a Stanley-Reisner ring of its fan:
\[
H^*_T(X_\Sigma)=\varprojlim \R[x_1,\ldots,x_r]/(I+K_m) = \R[x_1,\ldots,x_r]/I.
\]

\providecommand{\bysame}{\leavevmode\hbox to3em{\hrulefill}\thinspace}
\providecommand{\MR}{\relax\ifhmode\unskip\space\fi MR }
\providecommand{\MRhref}[2]{%
  \href{http://www.ams.org/mathscinet-getitem?mr=#1}{#2}
}
\providecommand{\href}[2]{#2}

\end{document}